\documentclass[12pt,reqno]{amsart}
\usepackage[T1]{fontenc}
\usepackage[utf8]{inputenc}
\usepackage[english]{babel}
\usepackage{extsizes}
\usepackage{amssymb,amsfonts,amsthm,amsmath}
\usepackage{graphicx}
\usepackage{color}
\usepackage{subcaption}
\usepackage[pagebackref=true]{hyperref}
\usepackage{hyperref}
\usepackage[shortlabels]{enumitem}
    \hypersetup{
          unicode   = true,
          colorlinks= true,
    }
\usepackage{cite}
\usepackage{setspace}
\AtBeginDocument{
\addtolength\abovedisplayskip{-0.3\baselineskip}
 \addtolength\belowdisplayskip{-0.3\baselineskip}
 \addtolength\abovedisplayshortskip{-0.3\baselineskip}
 \addtolength\belowdisplayshortskip{-0.3\baselineskip}
}
\usepackage[left=1 in,top=1 in,right=1 in, bottom=1 in]{geometry}
\numberwithin{equation}{section}
\definecolor{darkred}{rgb}{.70,.12,.20}

\definecolor{darkgreen}{rgb}{.20,.52,.14}

\definecolor{byz}{rgb}{.44,.16,.39}

\numberwithin{equation}{section}
\newtheorem{lemma}{Lemma}
\newtheorem{remark}{Remark}
\newtheorem{definition}{Definition}

\newtheorem{assumption}{Assumption}
\newtheorem{theorem}{Theorem}

\newtheorem{proposition}{Proposition}

\newtheorem{corollary }{Corollary}
\newtheorem{Lad-Ur}{Ladyzhenskaya-Uraltceva iterative Lemma}

\newtheorem{Main Findings}[theorem]{Main Findings} 
\usepackage{todonotes}

\title{Peaceman Well Block Problem For Time-Dependent  Flows of Compressible Fluid}
\author{
 A. Ibraguimov$^1$, E. Zakirov.$^2$, I. Indrupskiy$^2$, D. Anikeev$^2$, A. Zhaglova$^2$.
}

\date{}
\begin{document}
\maketitle
\vspace{-1cm}
{\begin{center}
\small{
$^{1}$ Texas Tech University(USA) and Oil and Gas Research Institute of RAS,Moscow
\\
 e-mail: \texttt{akif.ibraguimov@ttu.edu}
\\
$^2$ Oil and Gas Research Institute of RAS, Moscow
\\
e-mail: \texttt{ezakirov@ogri.ru}, \texttt{i-ind@ipng.ru} \texttt{anikeev@ogri.ru}, \texttt{azhaglova90@gmail.com}
}
\end{center}}
\begin{abstract}
\noindent
 We consider sewing machinery between finite difference and analytical solutions defined at different scales: far away and near the source of the perturbation of the flow.
One of the essences of the approach is that coarse problem and boundary value problem in the proxy of the source model two different flows. In his remarkable paper Peaceman propose a framework how to deal with solutions defined on different scale for linear  \textbf{time independent} problem by introducing famous, Peaceman well block radius. In this article we consider novel problem how to solve this issue for transient flow generated by compressiblity of the fluid. 
We are proposing method to glue solution via total fluxes, which is predefined on coarse grid and changes in the pressure, due to compressibility, in the block containing production(injection) well.  It is important to mention that the coarse solution "does not see" boundary.

From industrial point of view our report provide mathematical tool for analytical interpretation of simulated data for compressible fluid flow around a well in a porous medium. It can be considered as a mathematical "shirt" on famous Peaceman well-block radius formula for linear (Darcy) transient flow but can be applied in much more general scenario.

In the article we use Einstein approach to derive Material Balance equation, a key instrument to define $R_0$.    

We will enlarge Einstein approach for three regimes of the Darcy and non-Darcy flows for compressible fluid(time dependent):

$\textbf{I}. Stationary ;  \textbf{II}. Pseudo \ Stationary(PSS) ;  \textbf{III}.  Boundary \ Dominated(BD).$

Note that in all known authors literature, rate of the production on the well is time independent. 
\textbf{Our MB equation tuned to prove that  corresponding Peaceman well block radius for each of the regime of flow is time independent, and converge to Peaceman well block radius when exterior reservoir radius is vanishing.}        
For clarity we will first derive  Peaceman Well Block formula for each  regimes  of  the flows in 1-D  Euclidean case and then more difficult and more practical 2-D radial cases.

\end{abstract}
\tableofcontents
\vspace{-1cm}
\section{Introduction}
Peaceman well-block radius\cite{zia7} is routinely  used by engineers, but was not rigorously studied even for steady-state flows. Detailed review on fundamentals Peaceman well bock radius for linear and non-linear stationary flows in porous media was recently accepted for publication in  Applied and Computational Mathematics, an international journal and was published in the archive \cite{izia}.   Here we just want we often refer on  to mention that concept of equivalent well block radius was introduced in Russia(see \cite{zia1},and \cite{zia2}), but was not translated, and therefore was not cited in modern literature, as it often happen. 
In the basis of the idea behind Peaceman well block radius lay material balance equation which enable to sew analytical solution with simulated  one, and interpreted result of computed value of the pressure   in the block containing well.
In this section we will describe the paradigm for material balance(MB)  as an algebraic set of equations, and indicate our intended application.
 To introduce MB system of equation let first consider the finite set of dependant variables
 \begin{equation}\label{p-q-def}
     \mathcal{P}=\left\{p_{\pm r_0, 0}(s);p_{0,\pm r_0}(s);p_{\pm 1, 0}(s);p_{0,\pm 1}(s);q_x^{\pm}(s);q_y^{\pm}(s)\right\}.
 \end{equation}
 Let 
 \begin{equation}\label{K-Q-def}
     \mathcal{K}=\left\{K_x^{\pm};K_y^{\pm}\right\} \ \textnormal{and} \ \mathcal{Q}=\left\{Q_x^{\pm};Q_y^{\pm}\right\}.
 \end{equation}
 be inputs, which in this study are considered to be  constants. To make discussion more motivated we will highlight  intended application.  
 Consider diffusive process in the domain $\Omega \ni 0$ with source/sink - $0$, which is igniting process and  $\Omega_N=\sum_{i=1}^N B_i$ grid approximating $\Omega$. Let  $\Omega_N\supset B_0\ni 0$ to be  characterised  by blocks $B_i$,  of the "size" $\Delta$, and contains $0$, for all $i$.
Major assumption is that process of the transport and changes of the fluid is much "faster" than geological process, and therefore dependents of the $\mathbf{K}$ and $\mathcal{Q}$ to be ignored.
Assume that conductivity at  blocks of the interest  are fixed flow generated by source($well$ which fixed  and located in the  box $B(\Delta)$ of size   $\Delta,$ and this property holds  for each $\Delta$.
Let set $ \mathcal{P}$ contains only parameters defined only in center $B_0=B_{0,0}$( domain of the parameters $p_{\pm r_0, 0}(s);p_{0,\pm r_0}(s) , \cdots$ are in $B_0$) and nearest  surrounding four blocks  $B_{i,J}$( domain of the parameters $p_{\pm r_0, 0}(s);p_{0,\pm r_0}(s)$ are in $B_{\pm 1,0} \ B_{0, \pm 1} . $).
 Consider the filtration, which controlled by  Material Balance (MB) equation as an  algebraic equation  w.r.t. unknown variable $p_{a,b}(s)$, depending on parameter $s$ and  input variable $q_a^b(s)$ also depending on parameter $s$. $s$ is the model for time. System also featured by input parameter $\tau,$ which associate to changes in properties of variables $p$ on time interval $\left[s,s+\tau\right].$ This $\tau$ in some sens connect our equation to Einstein equation of material balance(see \cite{Einstein56}, \cite{ibr-isank-sob}), is  predefined and  set to be very small. 
 \begin{remark}
   Note that  Einstein equation of material balance is naturally stochastic, whether ours is deterministic. In spite of that we think that  Einstein's method can be extended  stochastic processes defined on THE stochastic grid. We will leav this for further research.   
 \end{remark}   
 Dependent variables fro set $\mathcal{P}$ and  $\mathcal{P}$ w.r.t. parameters $\mathcal{K}$,$\mathcal{Q}$, and $\tau$ are subject to algebraic equation:
\begin{align}\label{Mat-Bal-Alg}
&       \tau\cdot K_x^- \cdot \left(p_{-r_0,0}(s) - p_{-1,0}(s)\right)=\tau\cdot q_x^-(s)+Q_x^-\left(p_{-r_0,0}(s+\tau)-p_{-r_0,0}(s)\right)\\
&       \tau\cdot K_x^+ \cdot \left(p_{r_0,0}(s) - p_{1,0}(s)\right)=\tau \cdot q_x^+(s)+Q_x^+\left(p_{-r_0,0}(s+\tau)-p_{-r_0,0}(s)\right)\\
&       \tau \cdot K_y^- \cdot \left(p_{0,-r_0}(s) - p_{0,-1}(s)\right)=\tau \cdot q_y^-(s)+Q_y^-\left(p_{-r_0,0}(s+\tau)-p_{-r_0,0}(s)\right)\\
&       \tau \cdot K_y^+ \cdot \left(p_{0,r_0}(s) - p_{0,1}(s)\right)=\tau \cdot q_y^+(s) +Q_y^+\left(p_{0,r_0}(s+\tau)-p_{0,r_0}(s)\right)
     \end{align}

Denote:
\begin{equation}\label{qx+qy}
   q_x(s)=q_x^-(s)+q_x^+(s) \ \  q_y(s)=q_y^-(s)+q_y^+(s) ,Q_x=Q_x^-+Q_x^+ \ \  Q_y=Q_y^-+Q_y^+
\end{equation}
and
\begin{equation}\label{qx+qy}
   q(s)=q_x(s)+q_y(s) \ ; \ Q=Q_x+Q_y.
\end{equation}
Assume symmetry condition w.r.t.  $+$ and $-,$ which we state as follows: 
\begin{definition}\label{direc-sym} Symmetry structural constrains w.r.t. $+, \ -$.
\begin{enumerate}
    \item $K$ coefficient
    \begin{equation}\label{qx+-Kx+-sym}
 K_x^-=K_x^+ = K_x \ ; \  K_y^-=K_y^+=K_y     .
\end{equation}
\item $q$ parameter
\begin{equation}\label{q-sym}
q_x^-(s)=q_x^+(s)=\frac{q_x(s)}{2} \ ; \ q_y^-(s)=q_y^+(s)=\frac{q_y(s)}{2}   .
\end{equation}
\item $Q$ coefficient
\begin{equation}\label{Q-sym}
  Q_x^-=Q_x^+=\frac{Q_x}{2} \ ; \ Q_y^-=Q_y^+=\frac{Q_y}{2} .   \end{equation}
 \item $p$ variable w.r.t first index 
 \begin{equation}\label{p-x-dir}
    p_{-r_0,0}(s)= p_{r_0,0}(s) =p_{r_0}^x(s)\ ; \    p_{-1,0}(s)= p_{1,0}(s)=p_{1}^x(s) .  
 \end{equation}
  \item $p$ variable w.r.t second index
 \begin{equation}\label{p-y-dir}
 p_{0,-r_0}(s)= p_{0,r_0}(s)=p_{r_0}^y(s) \ ; \    p_{0,-1}(s)= p_{0,1}(s)=p_{1}^y(s) .   
 \end{equation}
 
\end{enumerate}

\end{definition}

Then from \eqref{Mat-Bal-Alg}, and some basic algebraic manipulations follows 
\begin{align}
&       \tau\cdot 2\cdot K_x \cdot \left(p_{r_0}^x(s) - p_{1}^x(s)\right)=\tau\cdot q_x(s)+Q_x\cdot 2\cdot\left(p_{r_0}^x(s+\tau) - p_{r_0}^x(s)\right),\label{MB-A-Sym-x}\\
&       \tau\cdot 2\cdot K_y \cdot \left(p_{r_0}^y(s) - p_{1}^y(s)\right)=\tau\cdot q_y(s)+Q_y\cdot 2 \cdot \left(p_{r_0}^y(s+\tau) - p_{r_0}^y(s)\right)\label{MB-A-Sym-y}.
     \end{align}

If one will assume that   $(p_{r_0}^y(s) - p_{1}^y(s)=0,$ $\left(p_{r_0}^y(s+\tau) - p_{r_0}^y(s)\right), \text{and} ,  q_y(s)=0$ then we will get a precursor  for 1-D MB which in the \underline{case of  symmetry} in $x$ - direction will take a form
\begin{equation} \label{MB-A-Sym-1-D}
    \boxed{\tau\cdot 2\cdot K_x \cdot \left(p_{r_0}^x(s) - p_{1}^x(s)\right)=\tau\cdot q_x(s)+Q_x \cdot 2\cdot \left(p_{r_0}^x(s+\tau) - p_{r_0}^x(s)\right).}
\end{equation}
 
As a precursor for  2-D MB which in \underline{case of  symmetry and isotropy} we will assume that    $p_{r_0}=p_{r_0}^x=p_{r_0}^y, \ \cdots$ and anisotropy: $K_x=K_y$ letting   $q(s)=q_x(s)+q_y(s)$ and $Q(s)=Q_x(s)+Q_y(s)$ we will take a form   

\begin{align}\label{Mat-Bal-Alg-Sym-in all- isotr}
\boxed{\tau\cdot 4\cdot K \cdot \left(p_{r_0}(s) - p_{1}(s)\right)=\tau\cdot q(s)+Q(s)\cdot 4 \left(p_{r_0}(s+\tau) - p_{r_0}(s)\right).}
\end{align}


In general setting algebraic variable(letter) of interest $p_i^{x,y,\cdots}$, $i=0,1,2,\cdots$ may depend on parameter, and it is common in algebraic geometry structure. 
Assume that $i=0,1$ then  using above arguments, we are considering       Algebraic Parametric Structure as a sewing machinery between numerical  and "analytical" solutions :
\begin{equation}\label{1-D MB}
    \tau\cdot\left(J_{1,0}^p\cdot\left(p_0(s)-p_1(s) \right) -I_q\cdot q(s)\right)=L_q^{p_0}\cdot\left(p_0(s+\tau)-p_0(s)\right) .\\
    \end{equation}
    Value of  the coefficients  and their dependence on input parameters can vary  depending  on the intended applications, dimension, geometry and dynamics of the process, discretization $\text{etc}.$ 
    \begin{remark}
        In a view of the algebraic structure equation of intended application $p_i(s)$ are dependant variables, \eqref{1-D MB} $q(s)$ is main \underline{  function defining process} , and three  others coefficients  $J_{1,0}^p,$ $I_q,$ and $L_q^{p_0}$ contains essential characteristics of the algebraic and geometrical structure of the media of the flow  and its discretization by domain $\Omega_N$.We will choose   this coefficients in next paragraph.    
    \end{remark}
    
\section{Motivation for  Material Balance Equations and Application to Numerical Scheme}
Consider flow in the reservoir $\Omega$, and corresponding mathematical model. 
Numerical simulator of the flow provide three basic information 
\begin{enumerate}
\item Geometric approximation of the domain of multi-component and multi-phase flow 
\item Numerical Value of the functions of interest in each block

\item Value of the parameters characterising domain w.r.t. chemical and physical of the fluids and media
\end{enumerate}
    To motivate the algebraic structure of MB \eqref{1-D MB} equation, consider orthogonal grid of dimension $M\times N$ and size $\Delta_x$ and $\Delta_y.$ Let $P_{(M,N)}$ be $M\times N$ matrix of the pressure value, with an elements $p_{i,j}(t),$ which associate  block $B_{i,j}.$ Assume that block $B_{0,0}$ contains source at center $(0,0)$, which generate differences  in the function $p_{i,j}(t).$ Here $D=\Omega\times (0.h)$ is 3-dimensional cylinder and there  non flow in $z$ direction. Assume   Green type  function $p(x,y,t)$ be a solution of the the basic modeling problem
    \begin{align}
     &L\cdot\frac{\partial p(x,y,t)}{\partial t }-J\cdot\left(\frac{\partial^2}{\partial x^2 }+ \frac{\partial^2}{\partial y^2 } \right)p= I\cdot\delta(x,y) \ \ \text{in} \ \ \left(\Omega \setminus (0,0)\right)\times (-\infty,\infty) \ ,   \\
      &B(p)=0 \ \text{on} \ \partial\Omega \times (-\infty,\infty).  
    \end{align}

        Here  $B$ to be boundary operator, which in our case will be Dirichlet, or Newman operator.  
    To approximate  function $p(x,y,t)$ consider finite different solution of the problem in rectangular domain
          \begin{align}
      &L\cdot\frac{ p_{i,i}(t+\tau)-p_{i,i}(t)}{ \tau }-\\ \nonumber
          &J\cdot\left(\frac{p_{i-1,j}(t)-2p_{i,j}(t)+p_{i+1,j}(t)}{\Delta_x ^2}+
      \frac{p_{i,j-1}(t)-2p_{i,j}(t)+p_{i,j+1}(t)}{\Delta_y ^2}\right)= \\ \nonumber
      &I\cdot\frac{\delta_{i,j}}{\Delta_x\cdot \Delta_y\cdot h} \ \ \text{in} \ \ \Omega_N \setminus (0,0) \ ,   \\
      &B(p)=0 \ \text{on} \ \partial\Omega_N \times (-\infty,\infty) \,\  
    \end{align}  
      or
\begin{align}\label{MB-fdif-dx-dy}
      &L\cdot(\Delta_x\cdot \Delta_y\cdot h)\cdot  \left( p_{i,j}(t+\tau)-p_{i,j}(t)\right)=\\ \nonumber
          &\tau \cdot\left[J h\left(\frac{\Delta_y}{\Delta_x}(p_{i-1,j}(t)-2p_{i,j}(t)+p_{i+1,j}(t))+
      \frac{\Delta_x}{\Delta_y}( p_{i,j-1}(t)-2p_{i,j}(t)+p_{i,j+1}(t))\right)+I\delta_{i,j}\right] ,\\ \nonumber          &B(p)=0 \ \text{on} \ \partial\Omega_N \times (-\infty,\infty).  
    \end{align}  
  Here $\delta_{i,j}$ is Kronecker symbol.
Equation above is basic and can be applied in $1-D$ and $2-D$ cases, although has in both cases many similarities but it differ.  Namely: 
\begin{enumerate}

    \item \underline{  1-D Material Balance in "last blocks" $B_0, B_1$}
    
    Under assumption   \underline{1-D Symmetry } let $\Delta_y=const$, and $h=const$ for all $\Delta=\Delta_x$ then MB  takes form 
    \begin{align}\label{MB-fin-dif}
      &L\cdot\Delta\cdot \Delta_y h \left( p_{0}(t+\tau)-p_{0}(t)\right)=\\ \nonumber
          &\tau\left(2\cdot J\Delta_y\cdot h\cdot\frac{\left(p_{1}(t)-p_{0}(t)\right)}{\Delta}+I\delta_{0,0}\right)=\tau\left(2\cdot (J\cdot(\Delta_y\cdot h))\cdot\frac{\left(p_{1}(t)-p_{0}(t)\right)}{\Delta}+q\delta_{0,0}\right) ,\\ \nonumber         
    \end{align} 
    
    \item \underline{ Radial Material Balance Equation in "last blocks" $B_0, B_{pm 1, 0}, B_{0, pm 1} $ }
  
  Under assumption of $\underline{ 2-D \text{ symmetry}}$ if one will will let that  
  \begin{equation}\label{delta-sym}
  \boxed{\Delta=\Delta_x=\Delta_y}. 
  \end{equation}
  We will also assume that thickness  of the reservoir is constant and 
  \begin{equation}\label{q-ss-pss}
 \boxed{I=q}. 
  \end{equation}
Then equation \eqref{MB-fdif-dx-dy} can be simplified as
   \begin{align}\label{MB-fin-dif}
      &L\cdot\Delta^2\cdot h \left( p_{0}(t+\tau)-p_{0}(t)\right)=\\ \nonumber
          &\tau\left(4\cdot J\cdot h\cdot\left(p_{1}(t)-p_{0}(t)\right)+I\delta_{i,j}\right)=\tau\left(4\cdot (J\cdot h)\cdot\left(p_{1}(t)-p_{0}(t)\right)+q\delta_{i,j}\right) ,\\ \nonumber           
    \end{align} 
\end{enumerate}
For convenience let summarize  comments on properties of the media  w.r.t. the flow  as itemized remarks 
\begin{remark}
\begin{enumerate} Physical consideration
\item In above $q$ is total rate on the well(well production), which is time dependant $\left(q=q(s)\right)$ for PSS and BD regimes. 
\item In this article for 2-D flows we assume isotropy and symmetry flows: 
\begin{equation}\label{sym-aniz}
   \boxed{ K=K_x^-=K_x^+=K_y^-=K_y^+ \ ; \   q_x^-=q_x^+=q_y^-=q_y^+ \ ; \ Q_x^-=Q_x^+=Q_y^-=Q_y^+.}
\end{equation}
\item In this article for 1-D flows we assume isotropy and symmetry flows:
\begin{equation}\label{sym-aniz-1-D}
   \boxed{ K=K_x^-=K_x^+ \ ; \   q=q_x^-=q_x^+\  q_y^-=q_y^+=0 \ ; Q= Q_x^-=Q_x^+ \; \  Q_y^-=Q_y^+=0.}
\end{equation}
\item All above assumptions was made enable   use  analytical solution, which can be constructed explicitly.  In case when analytical solution is unavailable  in the gluing machinery one can use the numerical  solution on fine scale for the corresponding IBVP.  
\item MB equation "does not see boundary $\Omega_N$ and used to glue analytical solution by solving Peaceman problem. But analytical solution will take into account impact of boundary condition on value of Peaceman Radius. We will see that in linear case $R_0$ will depend on size of the domain only in time dependent problem. Once more $R_0$ as it was shown by Peaceman will be independent on size of the domain, and this is quite remarkable finding by Peaceman. This issue was in the detail discussed in our article \cite{izia}        
\end{enumerate}

 \end{remark}

\begin{subsection}{1-D Material Balance } 
  
    Now consider Grid given in Fig. \ref{Einstein-MB-1D} with no flow condition in $y$ direction. Assume $h\cdot\Delta_y=1$ and let $\Delta_x=\Delta$. Then $1-D$ approximation of the Grid as in Figure 2'. Corresponding MB on 1-D  grid in $y$ dirrection is considered to be "trivial" 

Let in material balance equation \eqref{1-D MB} 
    
    \begin{align}\label{notation-MB}
    &I_q(s)=q\cdot\frac{1}{h\cdot\Delta_y\cdot \Delta_x}, \\
    &J_{1,0}^p=2K\cdot\frac{1}{ \Delta_x^2}, \\
    &L_q^{p_0}=\phi\cdot C_p. 
    \end{align}

    In above we let 
    \begin{equation}\label{delta-yh=1}
        \Delta_y\cdot h=1, \ C^0=\phi\cdot C_p,q(s)=q_x(s).
    \end{equation}
    \begin{equation}\label{notation-delta}
        \Delta_x=\Delta.
    \end{equation}
    

   %
    Then MB equation \eqref{1-D MB} takes form
    \begin{equation}\label{1-D MB-t}
      2K\cdot\left(p_0(s)-p_1(s) \right) =-q\cdot\Delta+C^0\cdot\frac{p_0(s+\tau)-p_0(s)}{\tau}\cdot\Delta^2 .
    \end{equation}
    Note that if we will  use the finite difference\eqref{MB-fin-dif} as a MB  one will get
     \begin{align}\label{MB-fin-dif-1D}
      &L\cdot\Delta\cdot \Delta_y h \left( p_{0}(t+\tau)-p_{i,i}(t)\right)=\\ \nonumber
       & \tau\left(2\cdot (J\cdot(\Delta_y\cdot h))\cdot\frac{\left(p_{1}(t)-p_{0}(t)\right)}{\Delta}+q\delta_{i,j}\right) ,\\ \nonumber          &B(p)=0 \ \text{on} \ \partial\Omega \times (-\infty,\infty), 
    \end{align} 
    which is equivalent to \eqref{1-D MB-t} under \eqref{delta-yh=1},\eqref{notation-delta} if $J=K \ ,\text{and} \, L=C^0 .$
\end{subsection}
    
\begin{subsection}{Linear Steady State Material Balance(Algebraic) $p$ does not  depend on $s$.} 

Constrain that $p_i$ for all $i$-s are  $s$ independent can be changed to more physical one:
 One of the multiplier in  equation \eqref{SS-contstrain} is equal zero, namely

 \begin{equation}\label{SS-contstrain}
   \phi C_p\frac{V_0}{V}\cdot\frac{p_0(s+\tau)-p_0(s)}{\tau} \equiv 0 .
 \end{equation}

Physically \eqref{SS-contstrain} it can be stated that flow of fluid or fluid itself is such that if one can assume that

\begin{enumerate}
\item $\tau$ - the time interval of compression of the fixed volume $V_0$ with respect to whole volume $V$ of flow filtration is too big;
\item porosity $\phi$ is negligible ;
\item compressibility $C_p$ is negligible;
\item changes of the pressure in the block $V_0$ w.r.t. $\tau$ is negligible.
\end{enumerate}
\begin{remark}
     Although we consider fracture $ \frac{V_0}{V}=1$ \ we keep this factor for interpretation  in \eqref{SS-contstrain}for mathematical generality.
 \end{remark}
\begin{definition}\label{ss-MB-1-D}
We will say that MB balance is steady state if condition \eqref{SS-contstrain}  holds for all $s$ and $\tau.$ 
 \end{definition}
Then symmetrical, isotropic and steady-state steady state 1-D  Balance Equation  has the form  
\begin{equation}\label{ss-MB-1D}
    2K\cdot(p_0-p_1)=q\Delta .
\end{equation}

\begin{remark}
Note that in steady state MB $p_i$ is parameter(time in our intended application) independent. 
\end{remark}
To sew value of $p_0$ with pressure trace on the boundary of flow with given rate $q$ let us   consider 1-D flow of non-compressible fluid towards  gallery $x=0.$ The flow is  subjected to: (i) linear Darcy equation with fixed,  (ii) $s$ independent  pressure $p=p_e$ on the reservoir boundary $x=r_e,$ and (iii) production rate $q$ on the gallery  at $x=0$ as the inner boundary of the flow. Corresponding 
analytical model for $1-D$ pressure has a form 
\begin{align}
  \frac{d}{dx}\left(K \frac{d}{dx}p_{an}(x)\right)=0;\label{ss-eq}\\
   p_{an}(x)\Big|_{x=R_e}=p_e;\label{ss-bc-ext}\\
   -K\cdot
\frac{d p_{an}}{dx}\Big|_{x=0}=q .\label{ss-well-cond}
\end{align}
     
\begin{definition}\label{peac-ss-defin}
Let $1-D$ domain $(0;R_e)$ is split by  grid $\left[0,\Delta,2\cdot\Delta,\cdots N\cdot\Delta\right]$ were $N\Delta =r_e.$ 
  We will say that Peaceman problem is well posed w.r.t. MB \eqref{ss-MB-1D} for $1-D$  flows if for any given $\Delta$ exist $R_0$ depending on $\Delta$ s.t. analytical solution of the 1-D SS problem \eqref {ss-eq}-\eqref{ss-well-cond} satisfies equation.   
\end{definition}

\begin{equation}\label{MB-an-1D}
    -2K\cdot\left(p_{an}(\Delta)-p_{an}(R_0)\right)=q\cdot \Delta.
\end{equation}

\begin{theorem}
In order  Peaceman problem to be  well posed w.r.t. MB \eqref{ss-MB-1D} for $1-D$  flows
it is necessary and sufficient that 
    \begin{equation}\label{R_0-SS}
       R_0=\frac{\Delta}{2}. 
    \end{equation}
    \begin{proof}
        First analytical  solution $p_{an}(x)$ has a form
\begin{align}
 &p_{an}(x)=A\cdot x+ B , \ A=-\frac{q}{K}, \  
  B=p_e + \frac{q}{K} r_e. 
\end{align}
   Substituting $p_{an}(x)$ into \eqref{MB-an-1D}one will get 
   \begin{align}\label{ss-R-0-calc}
   & 2KA((\Delta)-R_0+(B-B))=q\Delta \ \text{or}, \          2(\Delta)=\Delta + 2R_0 \ \text{or}, \         \Delta=2R_0.
   \end{align}
  From above formula \eqref{R_0-SS} follows.
 
    \end{proof}
      
\end{theorem}
\
\begin{remark}
This theorem is elementary  but we brought it here to highlight reasoning why Peaceman formula for $R_0$ for SS regime does depend only on size of the block $\Delta$  but does not  depend on size of the domain, conductivity and rate of production. In in-fact follows from mean-value Lagrange theorem, Darcy Law and Divergence Theorem (Conservation law) for not-compressible Fluid.      
\end{remark} 
\end{subsection}

\begin{subsection}{1-D Linear PSS  Material Balance(Algebraic) and corresponding $R_0$}

Let define the PPS  constrain for the solution of algebraic  material balance equation assuming in addition that $p_0$ $p_1$ and $q$ are conditioned  as follows

\begin{definition}\label{PSS-const4MB}
 We will say that MB is steady state if
  \begin{equation}\label{q-s-ind}
      q(s)=q \ \text{is $s$ independent.} 
  \end{equation}  
  \begin{equation}\label{p1-p0-s-ind}
      p_0(s+\tau)-p_0(s)=q\cdot C_0 \cdot \tau .\ \text{and } \  C_0\text{  is $s$ independent.}
  \end{equation}
  \end{definition}
  From above obliviously follows that  difference
  
   \begin{equation}\label{p0(s+tau)-p0(s)-tau}
       p_1(s)-p_0(s)\ \text{is $s$ independent}.   
   \end{equation}
Then linear 1-D PSS Material Balance  will have form 
  \begin{equation}
          2K\cdot\left( p_1- p_{0}\right)
          =q\cdot\Delta\left(1-\phi c_p\cdot 1 \cdot C_0\Delta \right)=q\Delta\left(1-C_1\Delta \right).
      \end{equation}
  For simplicity we will let $C_1=1$
  \begin{subsubsection}      
 {Analytical model for PSS problem $1-D$} 
  
  Analytical model for the PSS regime has a form

\begin{align}
  \frac{\partial}{\partial x}\left(K \frac{\partial }{\partial x}p(x,t)\right)=
  \frac{\partial p}{\partial t}; \ \text{on} \ (0;r_e)\times(0,\infty) \label{pss-eq-t}\\
 \frac{\partial  p}{\partial x}\Big|_{x=r_e}=0\label{pss-bc-ext-t}\\
   -K\cdot\frac{\partial p_{an}}{\partial x}\Big|_{x=0}=q\label{pss-well-cond-t}\\
   p(x,0)=p_{an}(x)\label{pss-IC}
\end{align}

In \eqref{pss-IC} $p_{an}(x)$is solution BV problem

\begin{align}
  \frac{d}{dx}\left(K \frac{d}{dx}p_{an}(x)\right)=Q=\frac{q}{r_e}; \ \text{on} \ (0;r_e) \label{pss-eq}\\
 K\cdot
\frac{d p_{an}}{dx}\Big|_{x=r_e}=0\label{pss-bc-ext}\\
   -K\cdot
\frac{d p_{an}}{dx}\Big|_{x=0}=q\label{pss-well-cond}
\end{align}

  It is evident that analytical solution for PSS IBVP has a form  

  \begin{equation}\label{pss-1-d-sol}
  p_{pss}(x,t)=w(x)+A_0 t, \ \ \text{function} \ w(x)=p_{an}(x). \end{equation}
  were $A_0=\frac{1}{r_e}\cdot q$.
 Solution  $p_{an}(x)$ of BVP \eqref{pss-eq}-\eqref{pss-well-cond}  is denoted as $w(x)$ for convenience.   PSS regime, generate  pressure $p_{ss}(x,t)$ is called pss-pressure. 
 
 General solution for $p_{an}(x)$ has a form 
 \begin{equation}\label{p_an}
p_{an}(x)=w(x)=Ax^2+Bx.     
 \end{equation}
 Here

 \begin{equation}\label{B-PSS-form}
B= \frac{q}{K} ,
 \end{equation}
\and is recovered  from BC \eqref{pss-well-cond}.
Parameter  $A$ consequently from RHS of \eqref{pss-eq} and BC \eqref{pss-bc-ext} 
 
 \begin{equation}\label{A-PSS-form}
     A=\frac{B}{2r_e}= - \frac{q}{2Kr_e}.
 \end{equation}
\begin{remark}
    Note that auxiliary function $w(x)$ by construction is vanishing on the well at $x=0$. 
\end{remark}
 \begin{definition}\label{peac-pss-defin}
  We will say that Peaceman problem for PSS is well posed w.r.t. time dependent MB \eqref{1-D MB-t} for $1-D$  flows if for any given $\Delta,$ and $r_e$ exist $R_0^{pss}(\Delta,r_e)$ depending on $\Delta$ and $r_e$ s.t. analytical solution of the 1-D PSS problem \eqref {pss-eq}-\eqref{pss-well-cond} satisfies equation and in addition  constrains \eqref{q-s-ind} and \eqref{p1-p0-s-ind} hold.   
\end{definition}
 \begin{theorem}\label{R-0-pss-peaceman}

     Peaceman problem is well posed w.r.t. time dependent MB for $1-D$, a.e. exists $R_0^{pss}(\Delta,r_e)$ s.t. $p_1(t)=p(\Delta,t)$, $p_0(t)=p(R_0,t)$ and $q$ s.t. $p(x,t)$ and $q$ a subject to time dependent  MB equation and both  constrains in the definition \ref{PSS-const4MB}.
     Moreover   the following limiting result follows
        \begin{equation}\label{limit-R-pss}
            \boxed{\lim_{r_e\to\infty}R_0^{pss}(\Delta,r_e) = R_0.}
             \end{equation}
     \end{theorem}
     \begin{proof}
      Proof follows from straightforward calculations that $R_0^{pss}(\Delta\,r_e)$ is subject to equation
      \begin{equation}
          2K\cdot\left( p_{pss}(\Delta,t)- p_{pss}(R_0,t)\right)=2K\cdot
          \left( w(\Delta)- w(R_0)\right)=q\Delta\left(1-\phi c_p\cdot 1\cdot C_2\frac{\Delta}{r_e}\right).
      \end{equation}
      From explicit presentation for $w(x)$ follows that $R_0$ should be subject for equation
      \begin{equation}\label{R0-pss-1-D}
      \frac{\Delta^2}{2r_e}+\Delta - \frac{\left(R_0^{pss}\right)^2}{2r_e}-R_0^{pss}=\frac{\Delta}{2}-C_3\frac{\Delta^2}{2r_e},
      \end{equation}
      for constant $C_3$ depending on $c_p, \ C_2$ only.
      From above we will get explicitly formula for $R_0^{PSS}:$
      \begin{equation}\label{R0-pss-1-D-expl}
      (1+C_3)\frac{\Delta^2}{2r_e}+ \frac{\Delta}{2}=R_0^{pss}+\frac{\left(R_0^{pss}\right)^2}{2r_e}.
      \end{equation}
      Assuming in above for simplicity $C_3=0$ one can get

 \begin{equation}\label{quad-eq-R-pss}
 \frac{(R_0^{pss})^2}{1}+2r_e \cdot R_0^{pss}-\left(r_e\cdot \frac{\Delta }{1}+\frac{\Delta^2}{1}\right)=0.   \end{equation}

From here follows that positive branch of the root satisfies chain of the equations

\begin{equation}
R_0^{pss}=\frac{-2r_e+\sqrt{4r_e^2+4\Delta r_e+8\Delta^2}}{2}=\frac{\Delta }{1+\sqrt{1+\frac{\Delta}{r_e}+2\frac{\Delta^2}{r_e^2}}}+\frac{2\Delta^2}{r_e}. 
   \end{equation}\label{exp-R-pss-chain-1d}

From above  chain of equations one can get   statement of the Theorem \ref{R-0-pss-peaceman}.
\end{proof}

     The following theorem follows straight forward from implicit differentiation of \eqref{R0-pss-1-D}
\begin{theorem}\label{mon-pss-1-D}
    Let $\Delta$ and all parameters bur $r_e$ of the problem are fixed. Then for $r_e$ big enough $R_0^{pss}$ as a function of $r_e$ decreasing.
\end{theorem}

 \end{subsubsection} 
 \end{subsection}

\begin{subsection}{1-D linear MB-constrains for $BD$ regime and corresponding well box radius $R_0$ }
    MB for BD regime can be stated in form of the definition for convenience in terms of keys input  constrains on  $q(s), \ p_i(s) \ \text{and} \ p_0(s+\tau).$
    
    \begin{definition}\label{MB-BDD-constr}
    
    Algebraic MB-constrains for boundary dominated is stated as follows. Exist constants $Q_0, \ \mathbf{P}_1. \mathbf{P}_0,$ s.t. items below hold varables $p_i$ and $q$ in MB equation
        \begin{enumerate}
            \item \begin{equation}
                \frac{q(s)}{p_0(s)}=Q_0(r_e)
            \end{equation}
            in above $Q_0(r_e)$ is $\Delta$ and $s$ independent constant
            \item 
            \begin{equation}
                \frac{p_1(s)}{p_0(s)}=\mathbf{P}_1(\Delta,r_e)  
            \end{equation}
            in above $\mathbf{P}_1(\Delta,r_e)$ is constant depending on  $\Delta$ and $r_e$, but not $s$. 
            \item
            \begin{equation}
                \frac{p_0(s+\tau)}{p_0(s)}=\mathbf{P}_0(\Delta,r_e)  \frac{e^{-C(K,r_e)\cdot\tau}-1}{\tau}
            \end{equation}
            in above $\mathbf{P}_0(\Delta,r_e)$ is constant depending on  $\Delta$ and $r_e$, but not $s$.
        \end{enumerate}
    \end{definition}

    \end{subsection}

    \begin{subsubsection}{Boundary dominated analytical problem}
    Consider analytical problem
    \begin{align}\label{BDD-anal-prob}
      &  \frac{\partial}{\partial x }\left(K\frac{\partial }{\partial x}u_0(x,t)\right)=c_0\cdot\frac{\partial u_0(x,t)}{\partial t}\\
       & u(x,t)\Big|_{x=0}=0\\
        &K\frac{\partial u_0(x,t)}{\partial x}\Big|_{x=r_e}=0\\
        &u_0(x,0)=\phi_0(x) \  \text{here} \ \phi_0(x) \text{is first eigenfunction.} \\
& u_0(x,t)\Big|_{t=0}=\phi_0(x) \
    \end{align}

Assuming for simplicity that $c_0=1$ It is not difficult to prove the following 
\begin{proposition}\label{MB-4-ansol}
    
Let $u_0(x,t)$ be an analytical  solution  of IBVP \eqref{BDD-anal-prob}:
\begin{equation}\label{bd-sol-0}
    u_0(x,t)=e^{-K\lambda_0 t} \sin(\lambda_0 x).
\end{equation}
Define variable in MB equation as 
\begin{align}
p_0(s)= u_0(R_0,s) \label{p_0-BD} ; \\   
p_1(s)=u_0(\Delta,s) \label{p_1-BD} ;\\ 
q(s)=K\frac{\partial u_0}{\partial x}\Big|_{x=0}.\label{q-BD}
\end{align}
Then  all items in Definition \ref{MB-BDD-constr} well defined for specific constants $Q_0, \ \mathbf{P}_0, \ \mathbf{P}_1$ for any $R_0$ and  $\Delta,$ 
and
\begin{equation}\label{lambda-0}
 \lambda_0=\frac{\pi}{2\cdot r_e}   
\end{equation}
\end{proposition}
    
    \begin{remark}
   Note that all initial Data for all $\mathbf{three \ analytical \ problems }$ are  assigned in such way that corresponding  productivity index are time independent.   
   
    \end{remark}
 It  important to state that existence of the constants is of main interest, and it will addressed below. We brought the Proposition \ref{MB-4-ansol} in order to follow frame of the construction and motivate   Peaceman well-posedness as
 \begin{definition}\label{peac-BDD-well-pos-def}
  We will say that Peaceman problem for BD regime is well posed w.r.t. time dependent MB \eqref{BDD-anal-prob} for $1-D$  flows if for any given $\Delta,$ and $r_e$ exist $R_0^{BD}(\Delta,r_e)$ depending on $\Delta$ and $r_e$ s.t. analytical solution of the 1-D BD problem \eqref {BDD-anal-prob} satisfies equation and in addition  constrains in the definition \eqref{MB-BDD-constr}  hold.   
\end{definition}
 \begin{lemma}\label{BD-R_0-generic}
   
Assume that $R_0^{bd}<\Delta$ then   for Peaceman's well posed for BD regime of the filtration it is sufficient
 \begin{equation}\label{r-0-bd-geneic}
   \sin (\lambda_0 R_0^{bd})-\sin(\lambda_0\cdot \Delta)+\frac{\lambda_0\Delta}{2} =\sin(\lambda_0\cdot R_0^{bd})\cdot\frac{1}{2K}\cdot\frac{e^{-\lambda_0^2\tau}-1}{\tau},
\end{equation} 
Moreover as $r_e \to \infty ,$ and $\tau\to 0$ $R_{0}^{bd}(\lambda_0,\tau)$ converges to Peaceman steady state $R_0.$ 

    Here $\lambda_0$ is the first eigenvalue, and $R_0^{bd}$, which  deliver solution to transcendent equation \eqref{r-0-bd-geneic} defines by value of $\Delta$, but and addition depend on $r_e,$  and $\tau \,\ K\,\ c_p$.
 \end{lemma}
\begin{proof}
To proof the Theorem it is suffice write down solution of the problem  \ref{BDD-anal-prob} 
and calculate
\begin{equation}\label{bd-sol-q}
    q(s)=K\frac{\partial u_0(x,t)}{\partial x}\Big|_{x=0}=e^{-K\lambda_0 s}\lambda_0 \cos(\lambda_0 x)\Big|_{x=0} =K\cdot e^{-K\lambda_0 s}\lambda_0
\end{equation}

Here $\lambda_0$ is First eigenvalue. 
    
   For connivance let rewrite MB 

     \begin{equation}\label{1-D MB-t-BD}
      2K\cdot\left(p_0(s)-p_1(s) \right) =-q(s)\cdot\Delta+ 1\cdot\frac{p_0(s+\tau)-p_0(s)}{\tau}\cdot\Delta
    \end{equation}
   and let $R_0=R_0^{bd}<\Delta$ be unknown then
   \begin{align}
     &p_0(s)=e^{-K\lambda_0^2 s} \sin(\lambda_0 R_0^{bd}),\label{p_0-bd-an}\\
     &p_1(s)=e^{-K\lambda_0^2 s} \sin(\lambda_0^2\Delta) \label{p_1-bd-an}\\
     &p_0(s+\tau)=e^{-K\lambda_0^2( s+\tau)} \sin(\lambda_0 R_0^{bd}),\label{p_0-tau-bd-an}\\
     &q(s)=K\cdot e^{-K\lambda_0^2 s}\lambda_0\label{q(s)-BD-an}
   \end{align}
Let $\lambda=\lambda_0$ be the first eigenvalue.  
   Using \eqref{p_0-bd-an}-\eqref{q(s)-BD-an} MB \eqref{1-D MB-t-BD} takes form
    \begin{equation}\label{1-D MB-t-bd-an}
      2K\cdot e^{-\lambda^2 s}\left(\sin(\lambda R_0^{bd} -\sin(\lambda \Delta) \right) =-K\lambda\cdot\Delta e^{-\lambda^2 s}+ 1\cdot\sin(\lambda R_0^{bd})\cdot \frac{e^{-\lambda^2(s+\tau)}-e^{-\lambda^2s}}{\tau}\cdot\Delta.
    \end{equation}
    After some factoring by$e^{-\lambda^2 s}$ equation \eqref{1-D MB-t-bd-an} takes form
  \begin{equation}\label{1-D MB-t-bd-an-1}
      2K\cdot \left(\sin(\lambda R_0^{bd} -\sin(\lambda \Delta) \right) =-K\lambda\cdot\Delta +1\cdot 1\cdot\sin(\lambda R_0^{bd})\cdot \frac{e^{-\lambda^2\cdot\tau}-1}{\tau}\cdot\Delta
    \end{equation}  
    
    Dividing \eqref{1-D MB-t-bd-an-1} by $2K$ we will get  
\eqref{r-0-bd-geneic} in the main Theorem \ref{BD-R_0-generic} .

In order analytical to satisfy MB equation it is sufficient to assume that  $R_0^{bd}(\Delta,r_e,\tau)$ solve equation \eqref{r-0-bd-geneic}.
One can simplify above equation \eqref{1-D MB-t-bd-an-1}. Namely
\begin{equation}\label{1-D MB-t-bd-simp}
      2K\cdot 2 \left(\sin\frac{\lambda\cdot(R_0^{bd}-\Delta)}{2}\cos\frac{\lambda(R_0^{bd}+\Delta)}{2}\right) =
      -K\lambda\cdot\Delta + 1\cdot\sin(\lambda R_0^{bd})\cdot\lambda^2\cdot \frac{e^{-\lambda^2\cdot\tau}-1}{\lambda^2\tau}\cdot\Delta.
\end{equation}

As $\tau \to 0$ from \eqref{1-D MB-t-bd-simp} on can have
\begin{equation}\label{1-D MB--bd-simp1}
  4K\cdot  \left(\sin\frac{\lambda\cdot(R_0^{bd}-\Delta)}{2}\cos\frac{\lambda(R_0^{bd}+\Delta)}{2}\right) = 
      -K\lambda\cdot\Delta + 1\cdot\sin(\lambda R_0^{bd})\cdot\lambda^2\cdot(-1)\cdot\Delta.
\end{equation}
 Under assumption that $\lambda$ to be such that $\cos\frac{\lambda(R_0^{bd}+\Delta)}{2} \approx 1 $ Last equation provide the compact approximation for $R_0^{bd}$ computation is small enough: 
 \begin{equation}\label{1-D MB--bd-aprox-0}
     R_0^{bd}\approx\frac{\Delta}{2}-\frac{1}{2K}\lambda^2R_0^{bd}\Delta
 \end{equation}
 or equivalently
\begin{equation}\label{1-D MB--bd-aprox}
     R_0^{bd}\approx\frac{\frac{\Delta}{2}}{1+\frac{1}{2K}\lambda^2\Delta}.
 \end{equation} 
 Finally taking into account explicit formula for $\lambda=\frac{\pi}{2r_e}$ and assuming that $\frac{\Delta}{8K}\frac{\pi^2}{r_e^2}$is small enough  we will get the approximate formula for $R_0^{bd}$
 \begin{align}\label{1-D MB--bd-aprox-0}
   &\boxed{R_0^{bd}\approx\frac{\frac{\Delta}{2}}{1+\frac{1}{8K}\frac{\pi^2}{r_e^2}\Delta}}  
 \end{align}
 

\end{proof} 

\begin{remark}
It is evident that for small enough $\tau$ , small $\lambda \Delta$ and small fraction  $\frac{\phi c_p}{2K}$ the appropriate $R_0^{BD}$ which solves equation \eqref{1-D MB-t-bd-an-1} exists. 
\end{remark}
Using Lagrange mean value result similar to aove  theorem to conclude if reservoir is unbounded then our obtained Peaceman well block radius is the same as in steady state (Classical) case. We state formulation and proof of the the following theorem for future more generic implementation in a view of Landis's multidimensional mean-value theorem \cite{landis-book}.  
 \begin{theorem}
 Well block radius $R_0^{bd}(r_e,\Delta)$ which deliver Peaceman well posedness asymptotically converges $\frac{\Delta}{2}=R_0$  as $r_e\to \infty$for any fixed $\tau$.
 \end{theorem}
\begin{proof}
    First observe that from Lagrange mean value theorem follows that 
\begin{equation}\label{r-0-bd-geneic-1}
   \cos\xi\cdot \lambda \left(R_0^{bd}- \Delta\right)+\frac{\lambda\Delta}{2} =\sin(\lambda\cdot R_0^{bd})\cdot\frac{1}{2K}\cdot\frac{e^{-\lambda^2\tau}-1}{\tau},
\end{equation} 
were \begin{equation}\label{xi}
  \lambda R_0^{bd} <\xi<\lambda \Delta, \text{and consequently } \cos{\xi}=1+O(\lambda).  
\end{equation} 
After division by $\lambda$ in \eqref{r-0-bd-geneic-1} we will get 
  \begin{equation}\label{r-0-bd-geneic-2}
   \cos\xi\cdot\left(R_0^{bd}- \Delta\right)+\frac{\Delta}{2} =\lambda\cdot\sin(\lambda\cdot R_0^{bd})\cdot\frac{1}{2K}\cdot\frac{e^{-\lambda^2\tau}-1}{\lambda^2\tau},
\end{equation} 
or assuming as before that $\lambda$ to be such that $\cos\xi\approx 1$ 

\begin{equation}\label{r-0-bd-geneic-3}
R_0^{bd}-\frac{\Delta}{2}=\lambda\cdot\sin(\lambda\cdot R_0^{bd})\cdot\frac{1}{2K}\cdot\frac{e^{-\lambda^2\tau}-1}{\lambda^2\tau}+O(\lambda).
\end{equation} 
Evidently in RHS of \eqref{r-0-bd-geneic-4} term $-1\cdot\frac{e^{-\lambda^2\tau}-1}{\lambda^2\tau}=O(1)$ therefore statement of theorem follows \eqref{lambda-0} and equation \eqref{r-0-bd-geneic-4} below.
\begin{equation}\label{r-0-bd-geneic-4}
   R_0^{bd}-\frac{\Delta}{2}=O(\lambda).
\end{equation} 
 
\end{proof}

\end{subsubsection}

\begin{figure}
    \centering
    \includegraphics[scale=0.5]{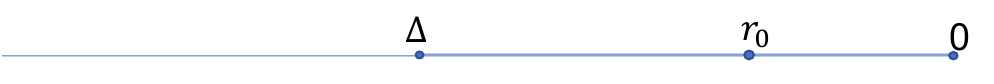}
    \caption{Einstein Mat Balance for 1-D Flow}
    \label{Einstein-MB-1D}
\end{figure}

 \begin{figure}
    \centering
    \includegraphics[scale=0.5]{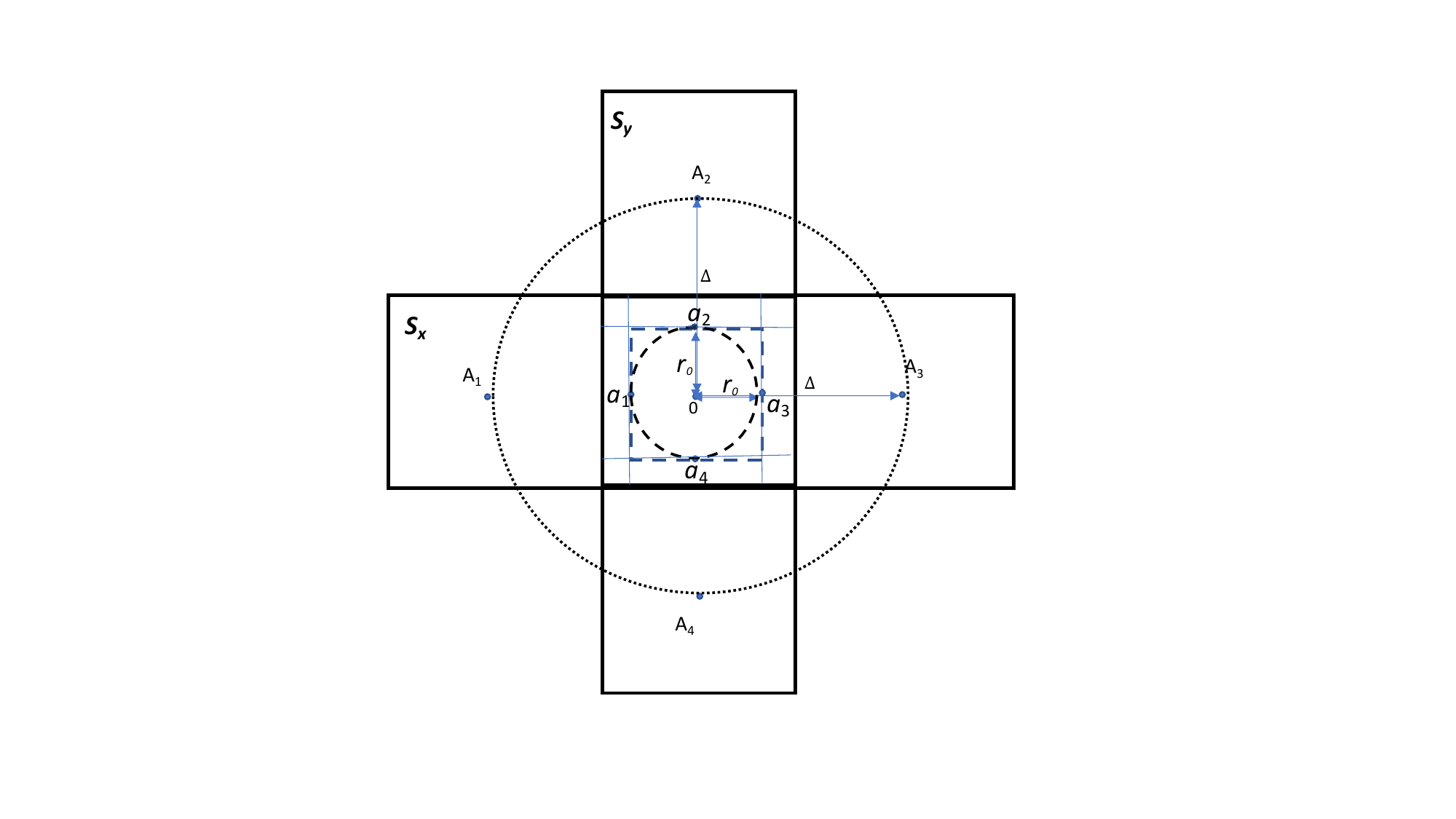}
    \caption{Einstein Mat balance equation on the 5 spots grid, 2-D case}
    \label{Einstein-mat-bal-r}
\end{figure}
\begin{section}{Pseudo Steady-State Material Balance }
Consider flow toward well $\Gamma_w$ in isolated reservoir $V$ of height $h=1$, and $\phi\cdot c_p =1$. Material balance equation for transient flow for slightly compressible fluid in a block with dimensions $\Delta \times \Delta \cdot 1$, volume $V_0=\Delta^2 \cdot 1$ and pressure $p_0$ containing a well (source/sink) with flow rate $q$ (positive for source and negative for sink):
\begin{equation}\label{basic-MB-PSS-Linear case-1}
    -4K\cdot(p_0(s)-p_1(s)) + \frac{q}{h} =  \Delta^2 \cdot 1 \cdot \frac {1}{\tau} \left(p_0(s+\tau)-p_0(s)\right), 
    \end{equation}

Let  the reservoir domain $U$ with volume V, boundary $\partial U =\Gamma_e \cup \Gamma_w $ and thickness $h$.

 \begin{assumption}
  Assume PSS constrain for slightly compressible fluid of compressiblity $c_p$.
 \begin{enumerate}
     \item 
      \begin{equation}\label{PSS-1-consrain}
        \left(p_0(s+\tau)-p_0(s)\right)= q \cdot \frac{\tau}{1 \cdot V }, 
          \end{equation}
          \item
          \begin{equation}
  \text{difference}    \ : \       p_0(s)-p_1(s)=constant \left(s \ \ \text{independent}\right).
          \end{equation}

 \end{enumerate}  
  Under above constrain in \eqref{PSS-1-consrain} in the assumption \eqref{PSS-MB} will take a form 
  \begin{equation}\label{basic-MB-PSS-Linear case}
    4K\cdot(p_0(s)-p_1(s)) = \frac{q}{1}  \cdot \left(1-\frac{\Delta^2}{V}\right),
    \end{equation}
  \underline{where $q$ is given constant in time   rate during given(fixed) time $\tau$, is considered to be the same} 
  
  \underline{for any time step $s$.} 
    
  \end{assumption}
  \begin{remark}
  Note that $p_i(s)$ depend on parameter (time in our application) $s$, but difference in PSS MB does not depend on $s$. It is remarkable difference in compare to Steady State regime.
  \end{remark}

Consider transient  2-D radial flow in the isolated the annual domain $U$for slightly compressible fluid towards well $\Gamma_w$ with given production rate  and non-flow condition on     the  radius $\Gamma_e$:

\begin{align}
&K\cdot \Delta p =1\cdot\frac{\partial p}{\partial t} \ \text{in} \  U=U(0,r_w,r_e) ; \label{ss-pde} \\
& K\cdot\frac{\partial p}{\partial \nu}=0\label{ext-BC} \ \text{on} \ \Gamma_e, \ \  r=r_e \ ;\\
&K\cdot\int_{\Gamma_w}\frac{\partial p}{\partial \nu}ds=-\tilde{q} \ \text{on} \ \Gamma_w \  , \  r=r_w .\label{well-BC} 
\end{align}
Here $$U(0,r_w,r_e)=\{x: r_w<|x|<r_e \} \ , \Gamma_w=\{x:|x|=r_w\ , \Gamma_e=\{x:|x|=r_e\} \, \ x=(x_1,x_2) ,$$ and 
$$\tilde{q}=\dfrac{q}{1}, V=1\cdot|U| \ , \  K=\frac{k}{\mu} , \ \frac{\partial p}{\partial \nu} \textnormal{ external derivative in co-normal direction  }  . $$

For generic case in order to deal with problem it is natural to consider the mixed  boundary value problem for elliptic equation which is well-defined  from mathematical point of view the following approach, and can be generalised for different scenarios.

In order $R_0$ to be time independent   we will split approach for IBVP. Namely consider   PSS solution of the above problem \eqref{ss-pde}-\eqref{well-BC}   defined as follows:
\begin{equation}
p_{pss}(x,t)=w(x)+At,    
\end{equation}
In above
\begin{equation}\label{A-definition}
  A=\frac{\tilde{q}}{1 \cdot |U|},  
\end{equation}

and $w(x)$ is solution of steady state problem:
\begin{align}
    \nabla \cdot \left(K  \nabla w(x)\right)= \frac{\tilde{q}}{|U|} \ \textnormal {in} \ U , \label{w-equation}\\
    w(x)=0 \ \text{on} \ \Gamma_w \ , \label{w-cond-well}\\
   K \frac{\partial w}{\partial \vec{\nu}}=0 \ \text{on} \ \Gamma_e. \ \label{ext-bound-cond}
\end{align}

In radial axial-symmetrical(radial flow) cases pseudo-steady state  $p_{pss}(r,t)$ will take a form 
\begin{equation}\label{pss-radial}
    p_{pss}(r,t)=w(r)+At
\end{equation}
Using representation for $ p_{pss}(r,t)$  It is not difficult to prove 
\begin{theorem}\label{PSS-MB}
In order  $\left[R_0^{PSS}(r_e,\Delta)\right]$ - Peaceman radius for PSS problem it is sufficient to find $\left[R_0^{PSS}(r_e,\Delta)\right]$  s.t.
\begin{equation}
4K\cdot\left(w(\Delta)-w(\left[R_0^{pss}(r_e,\Delta)\right])\right)=-\frac{ q}{1 \cdot V}\left( V-\Delta^2\right). \  
\end{equation}

\end{theorem}
From explicit form of the solution  of the problem\eqref{w-equation}-\eqref{ext-bound-cond} one can
find such $ \left[R_0^{pss}(r_e,\Delta)\right].$

We will use another approach based on the formulation of the problem in term of velocity field. This velocity framework will help in future to work on nonlinear flows and some cases  explicitly obtain associated  Peaceman well block radius. First let us start with 
\begin{remark}
  In generic set   $U$ be domain with  split boundary $\partial U=\Gamma_e\cup \Gamma_w,$ were $$\Gamma_e\cap \Gamma_w=\emptyset,$$ and $\Gamma_e, \ ,\Gamma_w$ are compacts.
\end{remark}

  we will say that velocity field $v(x)$ has PSS profile if following assumption holds 
    \begin{assumption}\label{assumpt-PSS-SC-fluid} 
  We will say that velocity field  subject for PSS regime of flows if velocity is time-independent and the   solves the following BVP
    \begin{align}
      \nabla \cdot \vec{v}=C \ \textnormal{in} \ U  \label{cont-eq} \\
     \vec{v}\cdot \vec{\nu}=0 \ \text{on } \Gamma_e \label{no-flow-BC}.
    \end{align}
      Note that due to divergence theorem constant $C$ form
      \eqref{cont-eq} satisfies
    \begin{equation}\label{div-th}
      \int_{\Gamma_w}  \vec{v}\cdot\vec{\nu}d s=\tilde{q}=C\cdot|U|  
    \end{equation}
    \end{assumption}
\begin{remark}
Note that    in our original research  PSS regimes\cite{ibragim-Prod-Ind} was defined in term of pressure function. Namely we assumed that flow is PSS if $\frac{\partial p }{\partial t}=constant,$ and no flow condition on exterior boundary. In our intended application both definitions are equivalent to each other.  
\end{remark}
\begin{theorem}\label{PSS-Radius}
 Exists solution of Peaceman problem for time dependant PSS regime of the production, and corresponding $R_0^{PSS}(r_e,\Delta)$    defined by the equation 
 \begin{equation}\label{R0-PSS-exact-3}
    -\pi+\frac{\left[R_0^{PSS}(r_e,\Delta)\right]^2}{r_e^2}+\pi\frac{ r_w^2}{r_e^2}=-2\cdot\left(\ln\frac{\Delta}{\left[R_0^{PSS}(r_e,\Delta)\right]}\right).
     \end{equation}
     Moreover 
     \begin{equation}
       \lim_{r_e\to \infty}  \left[R_0^{PSS}(r_e,\Delta)\right]=R_0^{SS}=R_{Peaceman}
     \end{equation}
\end{theorem}
\begin{proof}

   In general   vector field   solution of the BVP \eqref{cont-eq}-\eqref{no-flow-BC}  is not unique, but in radial case  is well defined :
 \begin{align}
     -\frac{1}{r}(r v(r))_r=C=\frac{\tilde{q}}{|U|} \ , \label{cont-eq-r}\\
     \vec{v}\cdot \vec{r}\big|_{r=r_e}=0 \ \label{no-flow-BC-r}
     \end{align}
Then for $r_w \leq r\leq r_e$:

 \begin{equation}\label{v-pss-an}
     v(r)=-\frac{C}{2}\cdot r +\frac{C_1}{r}.
 \end{equation}
 In above due to BC constants can be selected as:
 \begin{equation}\label{const-in-v-pss}
   C=\frac{q}{\pi(r_e^2-r_w^2)}=\frac{q}{|U|}, \ \text{and} \ C_1=\frac{C}{2}r_e^2.   
\end{equation}
 Then  due to Darcy equation  -  PSS solution  $w(r)=-\frac{\mu}{k}\int v(r) dr$ and consequently 
\begin{equation}\label{PSS-an-4-MB}
     w(r)=-K^{-1}\left[-\frac{C}{4}\cdot r^2 +C_1\ln r+C_2\right].
\end{equation}
 In above  constant $C_2$  have chosen s.t.  $w(r)|_{r=r_w}=0$ and therefore :
 \begin{equation}
     C_2=\left[ \frac{C}{4}\cdot r_w^2-C_1\ln r_w\right].
 \end{equation}
 
 Consequently expression for pressure function to choose $R_0^{PSS}(r_e,\Delta)$ using Theorem \ref{PSS-MB} in case of PSS  regime.

 \begin{equation}\label{PSS-an_MB}
      -\tilde{q}  \cdot \left(1-\frac{\Delta^2}{V}\right)=4K\cdot(p_1-p_0)=4K\cdot[w(\Delta)-w(R_0)]
 \end{equation}
 Then due to \eqref{PSS-an-4-MB} and \eqref{const-in-v-pss} one has 

 \begin{align}\label{PSS-anal-MB-anul-1}
     &-\tilde{q}  \cdot \left(1-\frac{\Delta^2}{|U|}\right)=4\cdot K K^{-1}\cdot\\ \nonumber
     &\cdot\left[\left(C\cdot\frac{\Delta^2}{4}-C_1\ln (\Delta)-C_2\right)-\left(C\cdot\frac{\left[R_0^{PSS}(r_e,\Delta)\right]^2}{4}-C_1\ln (\left[R_0^{PSS}(r_e,\Delta)\right])-C_2\right)\right] \\
     &=\left[C\cdot\left(\Delta^2-R_0^2\right)-4\cdot\left(C_1\ln \frac{\Delta}{\left[R_0^{PSS}(r_e,\Delta)\right]}\right)\right]=\nonumber\\
     &C\cdot\left[\left(\Delta^2-\left[R_0^{PSS}(r_e,\Delta)\right]^2\right)-2\cdot\left(r_e^2\ln \frac{\Delta}{\left[R_0^{PSS}(r_e,\Delta)\right]}\right)\right].\nonumber
 \end{align}
 After simplification one can get
 \begin{equation}\label{R0-PSS-exact-1}
    \Delta^2-|U|=\left(\Delta^2-\left[R_0^{PSS}(r_e,\Delta)\right]^2\right)-2\cdot\left(r_e^2\ln \frac{\Delta}{\left[R_0^{PSS}(r_e,\Delta)\right]}\right),
 \end{equation}
 or
 \begin{equation}\label{R0-PSS-exact-2}
  \left[R_0^{PSS}(r_e,\Delta)\right]^2-\pi\left(r_e^2-r_w^2\right)  =\left[R_0^{PSS}(r_e,\Delta)\right]^2-|U|=-2\cdot\left(r_e^2\ln \frac{\Delta}{\left[R_0^{PSS}(r_e,\Delta)\right]}\right).
 \end{equation}
 From later main equation \eqref{R0-PSS-exact-3} 
  follows.
 Second statement of the theorem follows follows  from   \eqref{R0-PSS-exact-3}.

\end{proof}

From Theorem \ref{PSS-Radius} formula \eqref{R0-PSS-exact-3}  follows monotonicity property of Peaceman PSS radius $R_0^{PSS}(r_e,\Delta)$ w.r.t. external radius $r_e$ 
\begin{theorem}\label{PSS-mon-radial-case}
    Under condition of applicability of our formulae  for PSS problem - $r_e\geq R_0^{PSS}(r_e,\Delta)$ 
    function $R_0^{PSS}(r_e,\Delta)$-monotonically decreasing.
\end{theorem}
\begin{proof}
  Taking derivative of left and right hand side of the equation \eqref{R0-PSS-exact-3} one can get 
  \begin{equation}\label{expl-der-pss-rad}
      2\frac{\left[r_e^2-\left(R_0^{PSS}(r_e,\Delta)\right)^2\right]}{R_0^{PSS}(r_e,\Delta)\cdot r_e^2}\cdot\frac{d}{d r_e}R_0^{PSS}(r_e,\Delta)=-2\left[\frac{\left(R_0^{PSS}(r_e,\Delta)\right)^2}{r_e^3}+\pi\frac{r_w^2}{r_e^3}\right]<0
  \end{equation}
  But in a view of applicability of the framework
  bracket $\left[r_e^2-\left(R_0^{PSS}(r_e,\Delta)^2\right)\right]>0$, therefor statement  of the theorem follows from   \eqref{expl-der-pss-rad}.
\end{proof}
\begin{remark}
    It is evident that as $r_e>>R_0$  one can get good approximation using Peaceman well block radius  
   
\begin{equation}\label{R0-PSS-approx}
    \frac{\pi}{2}\approx\left(\ln \frac{\Delta}{R_0}\right).
 \end{equation}
\end{remark}
\end{section}


\begin{section}{Linear Boundary Dominate Material Balance }
Linear Material Balance for slightly compressible fluid the same as before, and  flow toward well $\Gamma_w$ in isolated reservoir $V$ of height $h=1$, and $\phi\cdot c_p =1$.

Assume boundary dominated (BD) constrain for slightly compressible fluid the same as for PSS. Namely 
\begin{equation}\label{basic-MB-PSS-Linear case-1}
    -4K\cdot(p_0(s)-p_1(s)) + \frac{q(s)}{h} =  1\cdot \frac{V_0}{1}  \cdot \frac {1}{\tau} \left(p_0(s+\tau)-p_0(s)\right), 
    \end{equation}
   Once again let $V$ be the volume of the reservoir domain $U$ with boundary $\partial U =\Gamma_e \cup \Gamma_w $ .
 In order Peaceman radius for boundary dominated regime to be time independent assume the following
  \begin{assumption}\label{BD-assump for MB eq} 
Assume that:

   \underline{A-1}.
   
\begin{equation}\label{q/p1}
       \frac{q(s)}{p_1(s)}=C_1
\end{equation} 
   
Constant $C_1$ is $s$ independent

  \underline{A-2}.
  
   \begin{equation}\label{p0/p1}
       \frac{p_0(s)}{p_1(s)}=C_2
    \end{equation} 
   
  Constant $C_2$ is $s$ independent 
  
   \underline{A-3}. 
   
   \begin{equation}\label{p0s+/p0s}
      \tau^{-1}\cdot\left( \frac{p_0(s+\tau)}{\cdot p_0(s)}-1\right)\approx C_3  \ ,\ \text{for} \ \ \tau<<1.
   \end{equation} 
   
 Constant $C_3$ is $s$ and $\tau$ independent. 
 \end{assumption}  
 BD IBVP is defined as 
\begin{align}
&K\cdot \Delta p =c_p\phi\frac{\partial p}{\partial t} \ \text{in} \  U=U(0,r_e,r_w)=B(0,r_e)\setminus  B(0,r_w) ; \label{ss-pde} \\
& K\cdot\frac{\partial p}{\partial \nu}=0\label{ext-BC} \ \text{on} \ \Gamma_e, \ \  r=r_e \ ;\\
&p(x)=p_w \ \text{on} \ \Gamma_w \  , \  r=r_w .\label{well-BC} 
\end{align}
For simplicity we assume that $p_w=0$.

Once more assuming that radial flow towards well of  radius $r_w$, let  base solution of the problem above to be in the form of 
\begin{equation}\label{basic-anal-bdd-sol}
    p(x,t)=u_0(x,t)=e^{-\lambda_0\cdot t\frac{\cdot K}{1}}\varphi_0(x).
\end{equation}

Here $\varphi_0(x)$ first eigenfunction  and first $\lambda_0$ eigenvalue of the problem in the domain $U,$ with split boundaries $\partial U=\Gamma_w\cup \Gamma_e$.
\begin{align}
& -\Delta \varphi_0(x)=\lambda_0 \varphi_0(x) \ \text{in } \  U  \label{eigen-eq} ;\\
& \varphi_0(x)=0 \ \text{on} \ \Gamma_w \  ,\left( \text{in radial case when }  r=r_w \right)\label{eigen-well-cond} ; \\
& \frac{\partial \varphi_0(x)}{\partial \nu}=0\ \text{on} \ \Gamma_e. \ \left(\text{in radial case when }\  r=r_e \right)\label{eigen-ext-bound-cond}
\end{align}
\begin{remark}
    Motivation to consider this type of solution comes from  paper\cite{ibragim-Prod-Ind}, in which we proved that corresponding productivity index is time independent.  
\end{remark}

First let us  check   constrains in the \ref{BD-assump for MB eq} w.r.t. analytical solution \eqref{basic-anal-bdd-sol}

Indeed it is not difficult to see that all three conditions : $A_1$, $A_2$, and $A_3$ in Assumption \ref{BD-assump for MB eq} are satisfied  with 
\begin{align}
& C_1(\Delta_,R_0)=\frac{\varphi_0(\Delta)}{\varphi_0(R_0)} \
 C_2=\Lambda\frac{\int \varphi_0 dx}{\varphi_0(R_0)} \
C_3=\phi\cdot V_0 \cdot c_p \cdot \Lambda 
\end{align}
Finally assuming that $c_p=1$for given $\Delta$ one has equation for $R_0^{bd}$ for  
\begin{equation}\label{BD-eq-for-R0}
  \frac{\varphi_0(\Delta)}{\varphi_0(R_0)}+ \lambda_0\frac{\int \varphi_0 dx}{\varphi_0(R_0)}=\lambda_0\cdot V_0 ,\ \text{or} \ \frac{\varphi_0(\Delta)}{ \Delta^2}+ \frac{1}{c_p}\cdot\frac{\int_{\Gamma_w} \frac{\partial \varphi_0}{\partial \nu} ds}{\Delta^2}=\lambda_0\varphi_0(R_0).
\end{equation}
Function $\varphi_0(r)$ satisfying conditions \eqref{eigen-well-cond}-\eqref{eigen-ext-bound-cond} is a solution of the  Sturm-Liouville problem for the Helmholtz equation in an annual  domain with Dirichlet and Neumann conditions:
\begin{align}
  \frac{1}{r}\frac{\partial}{\partial r}\left( r \frac{\partial \varphi_0(r)}{\partial r } \right)+\lambda_0\varphi_0(r)=0, \ r_w<r<r_e\label{eq-SL} \\ 
     \varphi_0(r_w)=0, \frac{\partial\varphi_0(r)}{\partial{r}}\Big|_{r=r_e}=0 .\label{eq-SL-cond}
\end{align}

We are interesting in the non-negative solution of the above boundary value problem  which has the form (see \cite{Tich})

  \begin{equation}\label{varphi-r-sol}
    \varphi_0(r)={J_0(\sqrt{\lambda_0}{r_w})}{N_0(\sqrt{\lambda_0}{r})}-{N_0(\sqrt{\lambda_0}{r_w})}{J_0(\sqrt{\lambda_0}{r})},
\end{equation}
were   $\lambda_0$ to be first eigenvalue , which is solution of the transcendent equation: 

\begin{equation}\label{determ_0}
    {N_0(\sqrt{\lambda_0}{r_w})}{J'_0(\sqrt{\lambda_0}{r_e})}-{N'_0(\sqrt{\lambda_0}{r_e})}{J_0(\sqrt{\lambda_0}{r_w})}=0
\end{equation}

Consequently, the solution of the problem \eqref{ss-pde}-\eqref{well-BC} has a form
\begin{equation}\label{problem-sol}
    u_0(r,t)=e^{-\lambda_0\cdot t\frac{\cdot K}{1}}[{J_0(\sqrt{\lambda_0}{r_w})}{N_0(\sqrt{\lambda_0}{r})}-{N_0(\sqrt{\lambda_0}{r_w})}{J_0(\sqrt{\lambda_0}{r})}]
\end{equation}

One can directly verify all constrains in Assumption \ref{BD-assump for MB eq} by letting 
\begin{equation}
p_1(s)=u_0(\Delta ,s) \ , \ p_0(s)=u_0(R_0,s) \, \ p_0(s+\tau)=u_0(R_0,s+\tau) \, \  q(s)=-2\pi r_w\cdot K \frac{\partial u_0(r,s)}{\partial r}\big|_{r=r_w}
\end{equation}

In this article we will not provide proof of the existence of  Peaceman well block radius for boundary dominated regime in radial case, and investigate its property depending on the parameters of the problem. Instead of that we will  state the statement in form of remark leaving details for upcoming publication.   
 \vspace{-0.3 cm}
\begin{remark}
   Substitute  $p_0(s)$, $p_1(s)$ ,  $p_0(s+\tau)$, and $q(s)$ into material balance equation \eqref{basic-MB-PSS-Linear case-1}.
   Then we will get transcended equation for $R_0^{BD}(r_e,\Delta)$ of the form 
\begin{align}
    \varphi_0(R_0^{BD}(r_e,\Delta))-\varphi_0(\Delta)=-\frac{2}{\pi}\ln{\frac{\Delta}{R_0^{BD}(r_e,\Delta)}}
\end{align}

Here $\varphi_0(r)$ is the first  eigenfunction of the problem \eqref{eq-SL}-\eqref{eq-SL-cond}, which defined by equation \eqref{varphi-r-sol}.
\end{remark}
\end{section}

\bibliographystyle{plain}



\end{document}